\documentclass[a4paper,11pt]{article}
\usepackage[utf8]{inputenc}
\usepackage[T1]{fontenc}
\usepackage{lmodern}
\usepackage{microtype}
\usepackage[a4paper,textwidth=16cm,textheight=22cm,centering]{geometry}

\usepackage{mathtools}
\usepackage{amssymb,amsfonts}
\usepackage{amsthm}

\usepackage{setspace}
\usepackage{indentfirst}

\usepackage{cite}
\usepackage{xcolor}
\usepackage[ ]{hyperref}
\hypersetup{
	colorlinks=true,   
	linktoc=all,        
	linkcolor=red,     
	citecolor=blue,} 

\theoremstyle{plain}
\newtheorem{theorem}{Theorem}[section]
\newtheorem{lemma}[theorem]{Lemma}
\newtheorem{claim}[theorem]{Claim}
\newtheorem{fact}[theorem]{Fact}

\newtheorem{conjecture}[theorem]{Conjecture}
\newtheorem{observation}[theorem]{Observation}

\theoremstyle{definition}

\theoremstyle{remark}

\begin{document}
	\date{}
	\begin{spacing}{1.03}
		\title{{A double-exponential lower bound for $r_4(5,n)$ 
        }}
\author{
Longma Du$^1$, \;\;\; Xinyu Hu$^2$, \;\;\; Ruilong Liu$^1$, \;\;\; Guanghui Wang$^1$}

\footnotetext[1]{School of Mathematics, Shandong University, Jinan 250100, P.~R.~China.
Emails: {\tt 202520303@mail.sdu.edu.cn} (L. Du),
{\tt liuruilong@mail.sdu.edu.cn} (R. Liu),
{\tt ghwang@sdu.edu.cn} (G. Wang).
Supported by the Natural Science Foundation of China (12231018) and State Key Laboratory of Cryptography and Digital Economy Security.}
\footnotetext[2]{Data Science Institute, Shandong University, Jinan 250100, P.~R.~China.
Email: {\tt huxinyu@sdu.edu.cn} (X. Hu).
Supported by National Postdoctoral Fellowship Program (C-tier) (GZC20252005).}

\maketitle	

\begin{abstract}
The Ramsey number $r_k(s,n)$ is the smallest integer $N$ such that every $N$-vertex $k$-graph contains either a copy of $K_s^{(k)}$ or an independent
set of size $n$. 
We prove that $r_4(5,n)\ge 2^{2^{cn^{1/7}}}$, where $c>0$ is an absolute constant.  
As a consequence, we  determine the tower growth rate of $r_k(k+1,n)$, which completely solves the problem of establishing the tower growth rate for all classical off-diagonal hypergraph
Ramsey numbers, first posed by Erd\H{o}s and Hajnal in 1972. 

\end{abstract}
	
\section{Introduction}
A $k$-uniform hypergraph $H$ ($k$-graph for short) is a pair consisting of a set of vertices $V(H)$ and a collection of $k$-element subsets of $V(H)$. Let $K_n^{(k)}$ be the complete $k$-graph on $n$ vertices. The off-diagonal Ramsey number $r_k(s,n)$ \cite{R}, which has been extensively studied when $k$ and $s$ are fixed, is the smallest integer $N$ such that  every $N$-vertex $k$-graph contains either a copy of $K_s^{(k)}$ or an independent
set of size $n$. 

Off-diagonal Ramsey numbers have been extensively studied since 1935 \cite{E-S-1}, with many classical results \cite{A-K-S-1,A-K-S-2,B-K-2,C-G-E,E-H-Con,E-R-2,K-1,L-R-Z,M-S-3,M-S-4,S-1}.
In recent years, there have been many breakthroughs regarding off-diagonal Ramsey numbers, especially in graphs \cite{C-J-M-S,C-F-S-3,H-H-K-P,H-M-S,M-S-X,M-V-2}. Indeed, there have been significant recent breakthroughs in diagonal Ramsey numbers as well (see, e.g., \cite{C-G-M-S,G-N-N-W}). This work, however, focuses on the off-diagonal case, and diagonal Ramsey numbers are therefore not covered.

Campos, Jenssen, Michelen and Sahasrabudhe \cite{C-J-M-S} showed that $r_2(3,n)\ge (\frac{1}{3}+o(1))\frac{n^2}{\log n}$, which improved the results of Fiz Pontiveros, Griffiths and Morris \cite{P-G-M}. 
Shortly afterwards, Hefty, Horn, King and Pfender \cite{H-H-K-P} improved the lower bound $r_2(3,n)\ge (\frac{1}{2}+o(1))\frac{n^2}{\log n}$. Generally, the previous best bounds \cite{A-K-S-1,B-K-2,L-R-Z,SP-1} for $r_2(s,n)$ with fixed $s\ge 3$ are $\tilde{\Omega}(n^{\frac{s+1}{2}})\le r_2(s,n)\le\tilde{O}(n^{s-1})$. A breakthrough result of Mattheus and Verstra\"{e}te \cite{M-V-2} showed that $r_2(4,n)\ge\Omega(\frac{n^3}{\log^4n})$. 
In particular, Erd\H{o}s \cite{C-G-E} conjectured that $r_2(s,n)= {\tilde{\Theta}}(n^{s-1})$, where $\tilde{O}$, $\tilde{\Omega}$, and $\tilde{\Theta}$ suppress factors of the form $(\log n)^{\Theta(1)}$. In another recent development, Ma, Shen and Xie \cite{M-S-X} obtained the first exponential improvement for the lower bound of the off-diagonal Ramsey number $r_2(n,Cn)$ with constant $C>1$.

For $3$-graphs, Conlon, Fox and Sudakov \cite{C-F-S-3} obtained that for every $s\ge4$, $2^{\Omega(n\log n)}\le r_3(s,n)\le 2^{O(n^{s-2}\log n)}$, which improved the upper bound by Erd\H{o}s and Rado \cite{E-R-2} and the lower bound by
Erd\H{o}s and Hajnal \cite{E-H-Con}. 

For $s>k\ge 4$, it is known that $r_k(s,n)<\operatorname{twr}_{k-1} (n^{O(1)})$ \cite{E-R-2}, where the tower function $\operatorname{twr}_{k}(x)$ is defined by $\operatorname{twr}_{1}(x)=x$ and $\operatorname{twr}_{i+1}(x)=2^{\operatorname{twr}_{i}(x)}$. By applying the Erd\H{o}s-Hajnal stepping-up lemma, it follows that $r_k(s,n)\ge \operatorname{twr}_{k-1} (\Omega(n))$ for $k\ge 4$ and for all $s\ge 2^{k-1}-k+3$. A fundamental and important conjecture about $r_k(s,n)$ was proposed by Erd\H{o}s and Hajnal \cite{E-H-Con}.
\begin{conjecture}[Erd\H{o}s and Hajnal \cite{E-H-Con}]\label{E-H-C}
    Fix $4\le k<s$. It holds that
$r_k(s,n)\ge \operatorname{twr}_{k-1}(\Omega(n)).$
\end{conjecture} 
Erd\H{o}s and Hajnal (see \cite{G-R-S-1}) 
showed that $r_4(7,n)\ge 2^{2^{\Omega(n)}}$. In \cite{C-F-S-1}, Conlon, Fox and Sudakov modified the Erd\H{o}s-Hajnal stepping-up lemma to show that Conjecture \ref{E-H-C} holds for all $s\ge \lceil\frac{5k}{2}\rceil-3$. Mubayi and Suk \cite{M-S-4} as well as Conlon, Fox and Sudakov, independently verified Conjecture \ref{E-H-C} for $k\ge 4$ and $s\ge k+3$. Mubayi and Suk \cite{M-S-3} established the lower bounds $r_k(k+1,n)\ge \operatorname{twr}_{k-2} (n^{\Omega(\log n)})$ and $r_k(k+2,n)\ge \operatorname{twr}_{k-1} ({\Omega(n^{1/5})})$.  It remains an open problem to prove $r_k(k+1,n)\ge \operatorname{twr}_{k-1} (n^{\Omega(1)})$. 
In fact, this problem reduces to proving a double-exponential lower bound for $r_4(5,n)$ by a variant of the classical Erd\H{o}s-Hajnal stepping-up lemma\cite{M-S-3}.
\medskip

In this paper, we obtain a double-exponential lower bound for $r_4(5,n)$.

\begin{theorem}\label{center-2}
     For all $n\ge 5$, we have $r_4(5,n)\ge 2^{2^{cn^{1/7}}}$, where $c>0$ is an absolute constant.
\end{theorem}

As a consequence, this yields the lower bound $r_k(k+1,n)\ge \operatorname{twr}_{k-1}(c n^{1/7})$ for  $n>k\ge 5$, where $c=c(k)>0$ is a constant. This completely solves the problem of establishing the tower growth rate for all classical off-diagonal hypergraph Ramsey numbers. 

\medskip

The main idea in our approach is to apply stepping-up starting from a graph to construct a $4$-graph, and our construction relies on analyzing multi-layer extremum structures. The new ingredient has already been successfully applied to the Erd\H{o}s-Rogers problem\cite{D-H-L-W-1,D-H-L-W-2} and we expect more applications in the future.

\section{Properties of the stepping-up technique}\label{sec2}

Fix a positive integer $D$ and let $V = \{0, 1, \ldots, 2^D - 1\}$. For each $v \in V$, write
$v = \sum_{i=0}^{D-1} v(i)2^i$, where  $v(i) \in \{0, 1\}$ for every $i$.
For distinct $u,v\in V$, let
$\delta(u,v)$ denote the largest $i\in\{0,1,\ldots,D-1\}$ such that
$u(i)\neq v(i)$.

We use $\langle v_1, v_2,\ldots, v_r\rangle$ to denote an ordered vertex set with $v_1< v_2<\cdots<v_r$. For an ordered vertex set $S=\langle v_1, v_2,\ldots, v_r\rangle$, we also write
$\delta(S)=\delta(v_1,\ldots,v_r)
=(\delta(v_1,v_2),\ldots,\delta(v_{r-1},v_r))=(\delta_1,\ldots,\delta_{r-1})=(\delta_i)_{i=1}^{r-1}$, where $\delta_i=\delta(v_i,v_{i+1})$ for $i\in[r-1]$. 
For convenience, if inequalities are known between consecutive $\delta$-values, this will be indicated in the sequence by replacing the comma with the respective sign. For instance, assume that $S=\langle v_1,\ldots, v_5\rangle$ and $\delta_1< \delta_2 >\delta_3< \delta_4$. Then, since $\delta(v_1,v_2,v_3,v_4)=(\delta_1,\delta_2,\delta_3)$ satisfies $\delta_1< \delta_2 >\delta_3$, we write
$\delta(v_1,v_2,v_3,v_4)=(\delta_1<\delta_2>\delta_3)$.
Similarly, if not all inequalities are known, as in $\delta(v_1,v_2,v_4,v_5)$, we write $\delta(v_1,v_2,v_4,v_5)=(\delta_1<\delta_2~,~\delta_4)$.

We say that $\delta_i$ is a \emph{local minimum} if
$\delta_{i-1}>\delta_i<\delta_{i+1}$, a \emph{local maximum} if $\delta_{i-1}<\delta_i>\delta_{i+1}$, and a \emph{local extremum} if it is either a local minimum or a local maximum. We call $\delta_i$ a \emph{local monotone} if $\delta_{i-1}<\delta_i<\delta_{i+1}$ or $\delta_{i-1}>\delta_i>\delta_{i+1}$. We say $\delta_1,\ldots,\delta_{r-1}$
form a monotone sequence if $\delta_1<\cdots<\delta_{r-1}$ (monotone increasing) or $\delta_1>\cdots>\delta_{r-1}$
(monotone decreasing), i.e., there is no local extremum.

We have the following stepping-up properties, see \cite{G-R-S-1}.\medskip

\textbf{Property I.} For every triple $u < v < w$, $\delta(u,v) \neq \delta(v,w)$.
\medskip

\textbf{Property II.} For $v_1 < \cdots < v_r$, $\delta(v_1,v_r) = \max_{1 \leq j \leq r-1} \delta(v_j,v_{j+1})$.
\medskip

\begin{fact}\label{fact}
  Let $(\delta_{i_j})_{j=1}^{s}\subset (\delta_t)_{t=1}^{r-1}$ and  $\delta_{i_j}\neq\delta_{i_{j+1}}$ for all $j$. If $(\delta_{i_j})_{j=1}^{s}$ is a non-monotone sequence, then it contains a local extremum.
\end{fact}
\medskip

We also have the following properties from Properties I and II, see \cite{F-H-L-L,H-L-L-W-1,M-S-3}.
\medskip

\textbf{Property III.} For every 4-tuple $v_1 < \cdots < v_4$, if $\delta(v_1,v_2) > \delta(v_2,v_3)$, then $\delta(v_1,v_2) \neq \delta(v_3,v_4)$. Note that if $\delta(v_1,v_2) < \delta(v_2,v_3)$, it is possible that $\delta(v_1,v_2) = \delta(v_3,v_4)$.
\medskip

\textbf{Property IV.} For $v_1 < \cdots < v_r$, suppose that $\delta_1,\ldots,\delta_{r-1}$ form a monotone sequence. Then for every subset of $k$ vertices $v_{i_1} < \cdots < v_{i_k}$, $\delta(v_{i_1},v_{i_2}),\delta(v_{i_2},v_{i_3}),\ldots,\delta(v_{i_{k-1}},v_{i_k})$ form a monotone sequence. Moreover, for every subset of $k-1$ such $\delta_j$'s, i.e. $\delta_{j_1},\delta_{j_2},\ldots,\delta_{j_{k-1}}$, there are $k$ vertices $v_{i_1},\ldots,v_{i_k}$ such that $\delta(v_{i_t},v_{i_{t+1}}) = \delta_{j_t}$.
\medskip
\section{The lower bound for $r_4(5,n)$}

In this section, we begin by considering a graph coloring with specific properties, which will later be used to define the edge set of a 4-graph. 

\begin{lemma}\label{phi}
For every $n\ge 5$, there exists an absolute constant $c_0>0$ such that the following holds. There is a red/blue coloring $\phi$ of the pairs of $\{0,1,\ldots,\lfloor 2^{c_0n}\rfloor-1\}$ with the property that every $n$-set $A\subset \{0,1,\ldots,\lfloor 2^{c_0n}\rfloor-1\}$ contains a $3$-tuple $a_i<a_j<a_k$ satisfying $$\phi(a_i,a_j)=\phi(a_j,a_k)\ne\phi(a_i,a_k).$$
\end{lemma}

\noindent\textit{Proof of Lemma~\ref{phi}.~}Set $D=\lfloor 2^{c_0n}\rfloor$, where $c_0>0$ is a sufficiently small absolute constant to be determined later.
Consider a random red/blue coloring $\phi$ of the pairs of $\{0,1,\ldots,D-1\}$ in which each pair is colored independently and uniformly from the two colors.

We call a 3-tuple $a_i,a_j,a_k \in \{0,1,\ldots,D-1\}$ with $a_i < a_j < a_k $ \emph{good} if $\phi(a_i,a_j)=\phi(a_j,a_k)\ne\phi(a_i,a_k)$ and \emph{bad} otherwise.
For a fixed $3$-tuple, the probability that it is bad is $1-\frac{2}{2^3}=\frac{3}{4}$. 
Now fix an $n$-subset $A\subset \{0,1,\ldots,D-1\}$. We estimate the probability that $A$ contains no good $3$-tuple. Note that there exists a partial Steiner $(n,3,2)$-system $S$\footnote{That is, a $3$-graph on an $n$-vertex set with at least $c'n^2$ edges, for some constant $c'>0$, in which every pair of vertices is contained in at most one edge.}\cite{E-H-2}, on $A$ with at least $c'n^2$ edges. Since any two distinct $3$-tuples in $S$ share at most one vertex, they share no pair, and therefore the corresponding bad events are independent. It follows that the probability that all $3$-tuples in $S$ are bad is at most $\left(\frac{3}{4}\right)^{c'n^2}$. Hence, the probability that $A$ contains no good $3$-tuple is also at most $\left(\frac{3}{4}\right)^{c'n^2}$. 
Therefore, for a sufficiently small absolute constant $c_0>0$, the expected number of $n$-subsets $A$ containing no good $3$-tuple is at most $\binom{D}{n}\left(\frac{3}{4}\right)^{c'n^2}<\frac{1}{2}$. Thus, we conclude that there is a 2-coloring $\phi$ with the desired property. \hfill$\Box$
\medskip

Let $c_0>0$ be the constant from Lemma \ref{phi}, and let $U=\{0,1,\ldots,\lfloor2^{c_0n}\rfloor-1\}$ and $\phi:\binom{U}{2}\to \{\text{red},\text{blue}\}$ be a $2$-coloring of the pairs of $U$ satisfying the properties given in the lemma. Now, let $N=2^{\lfloor2^{c_0n}\rfloor}$ and $V(H)=\{0,1,\ldots,N-1\}$. Then we shall use the coloring
$\phi$ to produce a $K^{(4)}_5$-free $4$-graph $H$ on $V(H)$ with $\alpha(H)<2^7n^7+1$ as follows. For any $4$-tuple $e=\langle v_1,v_2,v_3,v_4 \rangle$ of $V(H)$, set $e\in E(H)$ if and only if one of the following holds:
\begin{enumerate}
\item[\textbf{(i)}] $\delta_1,\delta_2,\delta_3$ form a monotone sequence and $\phi(\delta_1,\delta_2)=\phi(\delta_2,\delta_3)\ne\phi(\delta_1,\delta_3)$.

\item[\textbf{(ii)}] $\delta_1>\delta_2<\delta_3$, $\delta_1>\delta_3$ and $\phi(\delta_1,\delta_2)=\phi(\delta_1,\delta_3)\ne\phi(\delta_2,\delta_3)$.
\item[\textbf{(iii)}] $\delta_1>\delta_2<\delta_3$, $\delta_1<\delta_3$ and $\phi(\delta_1,\delta_2)=\phi(\delta_1,\delta_3)=\phi(\delta_2,\delta_3)$.

\end{enumerate}

\subsection{$H$ is $K_5^{(4)}$-free }

In this subsection, we show that $H$ is $K_5^{(4)}$-free.
To see this, suppose to the contrary that $P = \langle v_1, \ldots, v_5 \rangle$ induces a $K_5^{(4)}$ in $H$. This will lead to a contradiction. Recall that $\delta(P)=(\delta_1,\delta_2,\delta_3,\delta_4)$.

\begin{claim}\label{no-mono}
$\delta(P)$ is not a monotone sequence.
\end{claim}
\noindent\textit{Proof of Claim \ref{no-mono}.~}Suppose, to the contrary, that $\delta(P)$ is increasing. The decreasing case is similar. Since $\delta(v_1,v_2,v_3,v_4)=(\delta_1<\delta_2<\delta_3)$ and \textbf{(i)}, we have $\phi(\delta_1,\delta_2)=\phi(\delta_2,\delta_3)\ne\phi(\delta_1,\delta_3)$. Similarly, since $\delta(v_2,v_3,v_4,v_5)=(\delta_2<\delta_3<\delta_4)$, we have $\phi(\delta_2,\delta_3)=\phi(\delta_3,\delta_4)$. Therefore, $(v_1,v_2,v_4,v_5)\notin E(H[P])$ since $\delta(v_1,v_2,v_4,v_5)=(\delta_1<\delta_3<\delta_4)$, and $\phi(\delta_3,\delta_4)\ne\phi(\delta_1,\delta_3)$ and \textbf{(i)}, a contradiction.\hfill$\Box$

It follows from Fact \ref{fact} and Claim \ref{no-mono} that there exists $i\in \{2,3\}$ such that $\delta_i$ is a local extremum. If $\delta_i$ is a local maximum, then we have $(v_{i-1},v_i,v_{i+1},v_{i+2})\notin E(H[P])$ since $\delta(v_{i-1},v_i,v_{i+1},v_{i+2})=(\delta_{i-1}<\delta_i>\delta_{i+1})$, a contradiction. Thus, we may assume that $\delta_i$ is a local minimum and we separate the proof into two cases as follows. 

\textbf{Case 1.~} $i=2$, that is $\delta(P)=(\delta_1>\delta_2<\delta_3<\delta_4)$. 

Since $(v_2,v_3,v_4,v_5)\in E(H[P])$ and $\delta(v_2,v_3,v_4,v_5)=(\delta_2<\delta_3<\delta_4)$, \textbf{(i)} gives $\phi(\delta_2,\delta_3)=\phi(\delta_3,\delta_4)\ne\phi(\delta_2,\delta_4)$. 
If $\delta_1<\delta_3$, then \textbf{(iii)} applied to $(v_1,v_2,v_3,v_5)$ gives $\phi(\delta_1,\delta_2)=\phi(\delta_2,\delta_4)$. Similarly, applying \textbf{(iii)} to $(v_1,v_2,v_3,v_4)$ gives $\phi(\delta_1,\delta_2)=\phi(\delta_2,\delta_3)$, contradicting $\phi(\delta_2,\delta_3)\ne\phi(\delta_2,\delta_4)$. 
If $\delta_1>\delta_4$, then \textbf{(ii)} applied to $(v_1,v_2,v_3,v_5)$ gives $\phi(\delta_1,\delta_2)\ne\phi(\delta_2,\delta_4)$. Similarly, applying \textbf{(ii)} to $(v_1,v_2,v_3,v_4)$ gives $\phi(\delta_1,\delta_2)\ne\phi(\delta_2,\delta_3)$. Since there are only two colors and $\phi(\delta_2,\delta_3)\ne\phi(\delta_2,\delta_4)$, this forces $\phi(\delta_1,\delta_2)=\phi(\delta_2,\delta_4)$, a contradiction. 
It follows from Property III that  $\delta_4>\delta_1>\delta_3$. Now $\delta(v_1,v_2,v_4,v_5)=(\delta_1>\delta_3<\delta_4)$ and $\delta_1<\delta_4$, so \textbf{(iii)} gives $\phi(\delta_1,\delta_3)=\phi(\delta_3,\delta_4)$. Since $\phi(\delta_3,\delta_4)=\phi(\delta_2,\delta_3)$, we get $\phi(\delta_1,\delta_3)=\phi(\delta_2,\delta_3)$. But applying \textbf{(ii)} to $(v_1,v_2,v_3,v_4)$ gives $\phi(\delta_1,\delta_3)\ne\phi(\delta_2,\delta_3)$, a contradiction.

\medskip

\textbf{Case 2.~} $i=3$, that is $\delta(P)=(\delta_1>\delta_2>\delta_3<\delta_4)$.

Since $(v_1,v_2,v_3,v_4)\in E(H[P])$ and $\delta(v_1,v_2,v_3,v_4)=(\delta_1>\delta_2>\delta_3)$, \textbf{(i)} gives $\phi(\delta_1,\delta_2)=\phi(\delta_2,\delta_3)\ne\phi(\delta_1,\delta_3)$. 
If $\delta_2>\delta_4$, then \textbf{(ii)} applied to the edge $(v_1,v_3,v_4,v_5)$ gives $\phi(\delta_3,\delta_4)\ne\phi(\delta_1,\delta_3)$, since $\delta(v_1,v_3,v_4,v_5)=(\delta_1>\delta_3<\delta_4)$ and $\delta_1>\delta_4$. Similarly, applying \textbf{(ii)} to $(v_2,v_3,v_4,v_5)$ gives $\phi(\delta_3,\delta_4)\ne\phi(\delta_2,\delta_3)$, which implies $\phi(\delta_3,\delta_4)=\phi(\delta_1,\delta_3)$, a contradiction. 
If $\delta_1<\delta_4$, then \textbf{(iii)} applied to $(v_1,v_3,v_4,v_5)$ gives $\phi(\delta_3,\delta_4)=\phi(\delta_1,\delta_3)$. Similarly, applying \textbf{(iii)} to $(v_2,v_3,v_4,v_5)$ gives $\phi(\delta_3,\delta_4)=\phi(\delta_2,\delta_3)$, contradicting $\phi(\delta_2,\delta_3)\ne\phi(\delta_1,\delta_3)$. 
Hence $\delta_1>\delta_4>\delta_2$. Now $\delta(v_1,v_2,v_4,v_5)=(\delta_1>\delta_2<\delta_4)$ and $\delta_1>\delta_4$, so \textbf{(ii)} gives $\phi(\delta_2,\delta_4)\ne\phi(\delta_1,\delta_2)$. On the other hand, applying \textbf{(iii)} to $(v_2,v_3,v_4,v_5)$ gives $\phi(\delta_2,\delta_4)=\phi(\delta_2,\delta_3)$. Since $\phi(\delta_1,\delta_2)=\phi(\delta_2,\delta_3)$, this is a contradiction.

\medskip

This completes the proof that $H$ is $K_5^{(4)}$-free.

\subsection{$\alpha(H)<2^{7}n^{7}+1$}
Now we show that $\alpha(H)<2^{7}n^{7}+1$. Suppose to the contrary that there is a set $Q=\langle v_1,v_2,\ldots ,v_m\rangle$ of $m=2^{7}n^{7}+1$ vertices that induces an independent set in $H$. Recall that $\delta_i=\delta(v_i,v_{i+1})$, and hence $\delta(Q)=(\delta_{i})_{i=1}^{m-1}$.

\begin{lemma}\label{no-mono-n}
There is no monotone subsequence $(\delta_{i_\ell})_{\ell=1}^n\subset (\delta_j)_{j=1}^{m-1}$ such that for any $a,b,c \in [n]$ with $a<b<c$, there exists $\{u_1,\ldots,u_4\}\subset\{v_1,\ldots,v_m\}$ such that $\delta(u_1,\ldots,u_4)=(\delta_{i_a},\delta_{i_b},\delta_{i_c})$.
\end{lemma}

\noindent\emph{Proof of Lemma \ref{no-mono-n}.~}
Without loss of generality, suppose to the contrary that $(\delta_{i_\ell})_{\ell=1}^n$ is such a monotone increasing subsequence. It follows from Lemma \ref{phi} that there is a $3$-tuple $\delta_{i_a},\delta_{i_b},\delta_{i_c}$ with $\delta_{i_a}<\delta_{i_b}<\delta_{i_c}$ such that $$\phi(\delta_{i_a},\delta_{i_b})=\phi(\delta_{i_b},\delta_{i_c})\neq\phi(\delta_{i_a},\delta_{i_c}).$$
Since there exists $\{u_1,\ldots,u_4\}\subset \{v_1,\ldots,v_m\}$ such that $\delta(u_1,\ldots,u_4)=(\delta_{i_a},\delta_{i_b},\delta_{i_c})$ from the assumption, $\{u_1,\ldots,u_4\}$ forms an edge by \textbf{(i)}, a contradiction.\hfill$\Box$
\medskip

Let $\beta_i=\frac{m-1}{(2n)^i}$, for $i\in[0,7]$. Since $m-1=(2n)^7$, each $\beta_i$ is an integer.
For $t\in[7]$, we will greedily construct \emph{$t$-layer local maxima sequences} $\Delta^{(t)}$ such that $\Delta^{(t)}\subset\Delta^{(t-1)}$, starting with $\Delta^{(0)}=\delta(Q)$, and such that the following property holds. (We do not require that the elements of $\Delta^{(t)}$ be distinct.)

\begin{itemize}
\item[($\ast$)] For two consecutive elements $\delta_a$, $\delta_b\in\Delta^{(t)}$, we have $\delta_x<\max\{\delta_a ,\delta_b\}$ for all $a<x<b$, and hence $\delta_a\neq\delta_b$.
\end{itemize}

For $t\ge1$, assume now that we have obtained $\Delta^{(t-1)}$ satisfying the desired property. We restrict our attention to $\Delta^{(t-1)}$ and we will find $\Delta^{(t)}$ to be the first $\beta_t$ local maxima (with respect to  $\Delta^{(t-1)}$) as follows. For convenience, we abbreviate ``with respect to'' as ``w.r.t.'' in the following. We claim first that there is no monotone consecutive subsequence of length $n$. Otherwise, suppose such a subsequence $Q'$ exists. Without loss of generality, we assume that $Q'$ is increasing. For any $\delta_{j_1},\delta_{j_2},\delta_{j_3}\in Q'$ with $j_1<j_2<j_3$, we have $\delta(v_{j_1},v_{j_1+1},v_{j_2+1},v_{j_3+1})=(\delta_{j_1}<\delta_{j_2}<\delta_{j_3})$ by noting the first part of the property ($\ast$) of $\Delta^{(t-1)}$ and Property II, which contradicts Lemma \ref{no-mono-n}. It follows from the second part of the property ($\ast$) of $\Delta^{(t-1)}$ and Fact \ref{fact} that we can set $\Delta^{(t)}$ to be the first $\frac{\beta_{t-1}}{2n}=\beta_t$ local maxima (w.r.t. $\Delta^{(t-1)}$). Therefore, $\Delta^{(t)}\subset \Delta^{(t-1)}$.

To show the  property ($\ast$) for $\Delta^{(t)}$,  we consider two consecutive elements $\delta_a$, $\delta_b\in\Delta^{(t)}$ and  we may assume that $\delta_{a},\delta_{i_1},\delta_{i_2},\ldots,\delta_{i_j},\delta_{b}$ are consecutive elements in $\Delta^{(t-1)}$. Note that $\delta_{a}$ and $\delta_{b}$ are consecutive local maxima (w.r.t. $\Delta^{(t-1)}$). We have $\delta_{i_\ell}<\max\{\delta_a, \delta_b\}$ for $\ell\in[j]$. Furthermore, it follows from the inductive hypothesis that $\delta_x<\max\{\delta_{i_\ell},\delta_{i_{\ell+1}}\}$ for all $i_\ell<x<i_{\ell+1}$ and $\ell\in[j-1]$, then $\delta_x<\max\{\delta_{i_\ell},\delta_{i_{\ell+1}}\}<\max\{\delta_{a},\delta_{b}\}$. Thus, $\delta_x<\max\{\delta_{a},\delta_{b}\}$ for all $a<x<b$. Moreover, Property I implies that $\delta_a\neq\delta_b$, as desired. Otherwise $\delta(v_a,v_b)=\delta_a=\delta_b=\delta(v_b,v_{b+1})$, a contradiction.
\medskip

Note that $\delta_j$ is a local maximum (w.r.t. $\Delta^{(t-1)}$) for $t\in[7]$ and $\delta_j\in\Delta^{(t)}\backslash\Delta^{(t+1)}$, where $\Delta^{(8)}=\emptyset$. We always let $\delta_{j^-}$ and $\delta_{j^+}$ be the \textbf{closest} element to the left and right of $\delta_j$ in the sequence $\Delta^{(t-1)}$, respectively. Thus,
$\delta_{j^-},\delta_{j^+}<\delta_{j}$. In particular, $\delta_{j^-},\delta_{j^+}\in \Delta^{(t-1)}\setminus\Delta^{(t)}$. From the above greedy construction, we can obtain the following observation by repeatedly using ($\ast$).

\begin{observation}\label{observation-1}
For $t\in [7]$ and  $\delta_j\in \Delta^{(t)}\backslash\Delta^{(t+1)}$,
we have $\delta_x<\delta_{j}$ for each $x\in[j^-,j^+]\setminus\{j\}$.
\end{observation}

Note that $|\Delta^{(7)}|=\beta_7=1$. Let $\Delta^{(7)}=\{\delta_a\}$. Define $\delta_{b_1}=\delta_{a^-}\in\Delta^{(6)}$, $\delta_{b_2}=\delta_{b_1^+}\in\Delta^{(5)}$ and $\delta_{b_3}=\delta_{b_2^+}\in\Delta^{(4)}$. 
Then $b_1<b_2<b_3<a$ and $\delta_{b_3}<\delta_{b_2}<\delta_{b_1}<\delta_a$. By the pigeonhole principle, there exist distinct $i,j\in[3]$ such that $\phi(\delta_{b_i},\delta_a)=\phi(\delta_{b_j},\delta_a)$. We give the details for the case
$\phi(\delta_{b_1},\delta_a)=\phi(\delta_{b_3},\delta_a)$. 
The other two cases 
are handled in exactly the same way: if 
$\phi(\delta_{b_i},\delta_a)=\phi(\delta_{b_j},\delta_a)$ with $i<j$, then one replaces 
$b_1$ and $b_3$ below by $b_i$ and $b_j$, respectively, and starts the subsequent 
construction from $b_j$. Now we let $\delta_c=\delta_{b_3^-}\in \Delta^{(3)}$, $\delta_d=\delta_{c^+}\in \Delta^{(2)}$, $\delta_e=\delta_{d^-}\in \Delta^{(1)}$ and $\delta_f=\delta_{e^+}\in \delta(Q)$. Recall that $Q=\langle v_1,v_2,\ldots ,v_m\rangle$ induces an independent set in $H$.

\begin{claim}\label{ind-1}
For every $c\le x<b_3$, we have $\phi(\delta_x,\delta_{b_3})\ne\phi(\delta_{b_3},\delta_a)$.
\end{claim}

\noindent\textit{Proof of Claim~\ref{ind-1}.~}
Suppose, to the contrary, that $\phi(\delta_x,\delta_{b_3})=\phi(\delta_{b_3},\delta_a)$ for some $c\le x<b_3$. We first claim that $\phi(\delta_x,\delta_a)=\phi(\delta_{b_3},\delta_a)$. Indeed, otherwise $\phi(\delta_x,\delta_{b_3})=\phi(\delta_{b_3},\delta_a)\ne\phi(\delta_x,\delta_a)$, which together with Observation \ref{observation-1} and \textbf{(i)} would imply that $(v_x,v_{x+1},v_{b_3+1},v_{a+1})\in E(H[Q])$, a contradiction.

Next, $\phi(\delta_{b_1},\delta_{b_3})\ne\phi(\delta_{b_3},\delta_a)$. Otherwise, since $\phi(\delta_{b_1},\delta_a)=\phi(\delta_{b_3},\delta_a)$, all three colors $\phi(\delta_{b_1},\delta_a)$, $\phi(\delta_{b_3},\delta_a)$ and $\phi(\delta_{b_1},\delta_{b_3})$ would be the same, which together with Observation \ref{observation-1} and \textbf{(iii)} would imply $(v_{b_1},v_{b_3},v_{b_3+1},v_{a+1})\in E(H[Q])$, a contradiction. Similarly, \textbf{(iii)} applied to $(v_{b_1},v_x,v_{x+1},v_{a+1})$ gives $\phi(\delta_{b_1},\delta_x)\ne\phi(\delta_x,\delta_a)=\phi(\delta_{b_1},\delta_a)$.

Since $\phi$ is a two-coloring, we get $\phi(\delta_{b_1},\delta_x)=\phi(\delta_{b_1},\delta_{b_3})\ne\phi(\delta_x,\delta_{b_3})$. \textbf{(ii)} then implies that $(v_{b_1},v_x,v_{x+1},v_{b_3+1})\in E(H[Q])$, a contradiction.
\hfill$\Box$

Without loss of generality, let $\phi(\delta_{b_3},\delta_a)=\text{red}$. Since $\phi$ is a two-coloring, Claim~\ref{ind-1} implies that $\phi(\delta_z,\delta_{b_3})=\text{blue}$ for every $z\in\{c,d,e,f\}$. 
Applying \textbf{(iii)} to $(v_c,v_d,v_{d+1},v_{b_3+1})$, the independence of $Q$ forces $\phi(\delta_c,\delta_d)=\text{red}$. The same argument applied to $(v_c,v_e,v_{e+1},v_{b_3+1})$, $(v_c,v_f,v_{f+1},v_{b_3+1})$ and $(v_e,v_f,v_{f+1},v_{b_3+1})$ gives $\phi(\delta_c,\delta_e)=\text{red}$, $\phi(\delta_c,\delta_f)=\text{red}$ and $\phi(\delta_e,\delta_f)=\text{red}$. Hence $\phi(\delta_c,\delta_d)=\phi(\delta_c,\delta_e)=\phi(\delta_c,\delta_f)=\phi(\delta_e,\delta_f)=\text{red}$. 
Next, applying \textbf{(ii)} to $(v_c,v_e,v_{e+1},v_{d+1})$ and $(v_c,v_f,v_{f+1},v_{d+1})$, respectively, the independence of $Q$ forces $\phi(\delta_e,\delta_d)=\phi(\delta_f,\delta_d)=\text{red}$. 
Thus $\phi(\delta_e,\delta_f)=\phi(\delta_e,\delta_d)=\phi(\delta_f,\delta_d)=\text{red}$. Since $\delta(v_e,v_f,v_{f+1},v_{d+1})=(\delta_e>\delta_f<\delta_d)$ and $\delta_e<\delta_d$, \textbf{(iii)} implies that $(v_e,v_f,v_{f+1},v_{d+1})\in E(H[Q])$, a contradiction.

This completes the proof of $\alpha(H)<2^7n^7+1$.

\medskip

Combining the two preceding subsections, we have constructed a $4$-graph $H$ on 
$N=2^{\lfloor 2^{c_0n}\rfloor}$ vertices such that $H$ is $K_5^{(4)}$-free and 
$\alpha(H)<2^7n^7+1$. Therefore 
$r_4(5,2^7n^7+1)>2^{\lfloor 2^{c_0n}\rfloor}$. Recall that $c_0>0$ is a sufficiently small absolute constant. Consequently, there exists an absolute constant $c>0$ 
such that $r_4(5,n)\ge 2^{2^{cn^{1/7}}}$ for all $n\ge 5$. This completes the proof of 
Theorem~\ref{center-2}.

\end{spacing}
\end{document}